\documentclass[11pt,letterpaper]{amsart}

\usepackage[english]{babel}
\usepackage[OT2,T1]{fontenc} 

\usepackage{mathtools}
\usepackage{amssymb}
\usepackage{stmaryrd}
\usepackage{bm}
\usepackage{tikz-cd}
\usepackage[version=4]{mhchem}

\usepackage{comment}
\usepackage{xcolor}
\usepackage{enumitem}
\usepackage{graphicx}
\usepackage{setspace}
\usepackage{url}

\usepackage[title]{appendix}
\usepackage{chngcntr}
\usepackage{apptools}
\usepackage{thmtools}

\usepackage[pagebackref]{hyperref}
\usepackage[nameinlink]{cleveref}
\usepackage{bookmark}

\definecolor{crimsonglory}{rgb}{0.75, 0.0, 0.2}
\definecolor{darkpowderblue}{rgb}{0.0, 0.2, 0.6}

\hypersetup{
	colorlinks = true,
	urlcolor   = darkpowderblue,
	linkcolor  = darkpowderblue,
	citecolor  = crimsonglory
}

\theoremstyle{plain}
\newtheorem{theorem}{Theorem}[section]
\newtheorem{definition}[theorem]{Definition}

\newtheorem{prop}[theorem]{Proposition}
\newtheorem{cor}[theorem]{Corollary}
\newtheorem{lemma}[theorem]{Lemma}
\newtheorem{claim}[theorem]{Claim}
\newtheorem{remark}[theorem]{Remark}
\newtheorem{remarks}[theorem]{Remarks}
\newtheorem{conj}[theorem]{Conjecture}

\DeclareMathOperator{\End}{End}
\DeclareMathOperator{\GL}{GL}

\DeclareMathOperator{\diag}{diag}

\DeclareMathOperator{\disc}{disc}
\DeclareMathOperator{\zar}{Zar}
\DeclareMathOperator{\chara}{char}
\DeclareMathOperator{\crys}{crys}

\DeclareMathOperator{\ord}{ord}
\DeclareMathOperator{\SP}{Sp}
\DeclareMathOperator{\adj}{adj}

\DeclareMathOperator{\ssing}{ssing}
\DeclareMathOperator{\good}{good}

\newcommand{\Sum}[2]{\displaystyle\sum_{#1}^{#2}}

\newcommand{\Z}{\mathbb{Z}}

\renewcommand{\C}{\mathbb{C}}
\newcommand{\Q}{\mathbb{Q}}

\newcommand{\CE}{\mathcal{E}}

\newcommand{\ld}{,\ldots,}

\newcommand{\CO}{\mathcal{O}}

\newcommand{\poly}{poly}
\newcommand{\Prob}{Prob}

\DeclareSymbolFont{cyrletters}{OT2}{wncyr}{m}{n}
\DeclareMathSymbol{\Sha}{\mathalpha}{cyrletters}{"58}

\newcommand{\mathsym}[1]{{}}
\newcommand{\unicode}[1]{{}}

\title{Lang-Trotter phenomena and unlikely intersections}
\author{Christopher Daw and Georgios Papas}
\begin{document}
			\begin{abstract}We show that the Lang–Trotter conjecture for pairs of elliptic curves implies new cases of the Zilber–Pink conjecture for curves in $\mathcal{A}_3$. Unlike previous results for curves in $\mathcal{A}_g$, our result does not rely on any assumption on intersections with the boundary, and in particular applies to potentially compact curves. The argument is based on the $G$-functions method of Yves André.
		\end{abstract}
			\maketitle

\section{Introduction}\label{section:intro}
Following the resolution of the Andr\'e-Oort conjecture by Pila-Shankar-Tsimerman in \cite{pilashankartsimerman}, the Zilber-Pink conjecture has become the central open problem in the area of unlikely intersections in Shimura varieties. While the Zilber-Pink conjecture remains widely open, several cases of it have been established for certain curves. 

In this direction, our main result is the following:
\begin{theorem}\label{mainthmzp}
	Let $S \subset \mathcal{A}_{3}$ be a smooth irreducible curve defined over $\bar{\Q}$ that is not contained in a proper special subvariety of $\mathcal{A}_{3}$.

Assume there exists a point $s_0\in S(\bar{\Q})$ whose corresponding abelian threefold $A_{0}:=A_{s_0}$ is isogenous to $E_{1} \times_{\bar{\Q}} E_{2} \times_{\bar{\Q}} E_{3}$ where
\begin{enumerate}
	\item $E_{1}$ is a CM elliptic curve,
	\item $E_{2}$ and $E_{3}$ are defined over $\Q$ and are not CM or isogenous over $\bar{\Q}$, and
	\item the Lang-Trotter conjecture for pairs holds for $\left(E_{2}, E_{3}\right)$.
\end{enumerate}

Then, the set
$$
\Sha(E^3,S):=\left\{s \in S(\mathbb{C}): A_{s} \text { is isogenous to } E^{3} \text { for some elliptic curve } E\right\}
$$
is finite.\end{theorem}

In contrast to previous Zilber-Pink-type results in $\mathcal{A}_g$ our \Cref{mainthmzp} does not depend on the existence of an intersection of our curve with the boundary in the Baily-Borel compactification. In particular, it applies to curves that are potentially compact in $\mathcal{A}_3$. Such curves do not arise in the case of $Y(1)^n$, where Zilber-Pink statements are known to hold under assumptions on how the curve $S\subset Y(1)^n$ intersects the boundary, see \cite{daworr4,papaszpy1,daworrpap}, or if they are ``asymmetric'' by work of Habegger-Pila \cite{habeggerpila1}.

\begin{remark}
\Cref{mainthmzp} is a first instance of a more general phenomenon, with related instances in the setting of $Y(1)^n$ established in \cite{papaszpy1note}. The assumptions on the fiber over $s_0$ reflect structural difficulties inherent in the G-functions method when combined with Lang-Trotter-type input, rather than the range of configurations to which these arguments can be expected to apply.
	
	The main role of the point $s_0$ is to provide a ``center'' for an associated family of G-functions. The splitting of the fiber $A_0$ over $s_0$ as a product of $3$ elliptic curves gives increased control over the periods of $A_0$. The assumption that $E_1$ is CM plays a twofold role. On the one hand, it further simplifies the periods of $A_0$. Most importantly, it ensures that places of bad reduction cannot occur as places of proximity between $s_0$ and points $s\in \Sha(E^3,S)$. Allowing such places to appear among the places of proximity would require a non-negligible extension of our exposition, without contributing to the new ideas introduced here.
	
	The assumption that $E_2$ and $E_3$ be defined over $\Q$ stems from the fact that, to the best of our knowledge, the Lang-Trotter conjecture for pairs has only been studied in this setting. Guided by Sato-Tate-type heuristics, in \Cref{section:langtrotternumberfield} we formulate a natural extension of this conjecture to number fields. Such a conjecture would remove the need for $E_2$ and $E_3$ to be defined over $\Q$. Finally, our arguments require a weaker form of this Lang-Trotter conjecture for pairs, see \Cref{remarkafterproof}.\end{remark}

\subsection{Large Galois Orbits and height bounds}\label{section:introlgoandheightbounds}

Following the Pila–Zannier strategy, the main difficulty in proving Zilber–Pink-type statements is to establish lower bounds for the size of Galois orbits. In \cite{daworr} the first named author and M. Orr first noticed a connection between the size of such Galois orbits and height bounds given by Andr\'e's G-functions method. Since then, Andr\'e's method has been the only tool consistently used to establish this missing piece in the Pila-Zannier strategy.

By a variant of the Pila-Zannier method due to D. Urbanik, see Proposition $7.27$ in \cite{urbanik}, our \Cref{mainthmzp} is reduced to the following height bound:
\begin{theorem}\label{heightboundintro}
	Let $S\subset \mathcal{A}_3$ and $s_0\in S(\bar{\Q})$ be as in \Cref{mainthmzp}. Let $h$ be a Weil height on $S$ and assume that the Lang-Trotter conjecture for pairs holds for the pair $(E_2,E_3)$. Then, there exist positive constants $c_0$, $c_1$ such that for all points $s$ in the set $\Sha(E^3,S)$ we have
	\begin{equation}
		h(s)\leq c_0 [\Q(s):\Q]^{c_1}.
	\end{equation}\end{theorem}

\subsubsection{Comparison with existing results} With the notable exception of \cite{habeggerpila1}, all previous Zilber-Pink-type results in the Shimura setting utilize Andr\'e's G-functions method to establish the crucial lower bounds for the size of Galois orbits. In all of these the point $s_0$ that appears in \Cref{heightboundintro} is chosen so that it reflects the existence of a degeneration of the geometric objects in the moduli space, or equivalently an intersection of the curve $S$ with the boundary in some compactification of the ambient Shimura variety.

This type of assumption has been persistently part of the G-functions method starting from Andr\'e's original work in \cite{andre1989g}. In the case of curves in $\mathcal{A}_g$ various versions of this assumption appear in \cite{daworr,daworr2,daworr3,daworr5,papaszp}. In the more general study of families of projective varieties similar assumptions are imposed on a point $s_0$ phrased in the language of mixed Hodge structures, see \cite{papasbigboi,urbanik,urbanikgalois}.

From the perspective of the G-functions method, the existence of a point $s_0$ on the intersection of the curve with the boundary enforces ``degeneration modulo $v$ conditions'' for points of interest $s$ for any finite place $v$. These conditions are then used to either outright rule out the $v$-adic proximity to $s_0$, as in \cite{andre1989g,papaszp} for example, or to construct relations that have no dependence on the place $v$, as is done for example in \cite{papasbigboi,daworr5,urbanik}.

The idea to consider instead $s_0$ to be a point that is inside the moduli space originated in work of Andr\'e \cite{andremots} and Beukers \cite{beukers}, where they both study the values of G-functions at CM points in $Y(1)$. This point of view has been taken up again in our recent work in \cite{daworrpap,papaspadicpart1,papaspadicpart2,papaspadicpart3}. 

Dropping the degeneration conditions on $s_0$, one loses a priori any uniform control on the set of finite places at which $s$ can be 
$v$-adically close to $s_0$, and as a result the ability to uniformly construct relations coming from global geometric data. Based on our previous work in \cite{daworrpap,papaspadicpart1,papaspadicpart2,papaspadicpart3}, for places $v$ of non-supersingular reduction of $A_0$ one can still hope to construct relations that are uniform in $v$. The essential difficulty we are left with is thus controlling the contribution from places of supersingular reduction of $A_0$. It is precisely here that we are led to invoke the Lang–Trotter conjecture for pairs of elliptic curves, see \Cref{langtrottersupersingularpairs}.

\subsection{Outline of the paper}

We begin in \Cref{section:periodmatrices} with a brief review of period matrices in both the archimedean and non-archimedean setting. Following the exposition in \cite{papaspadicpart1}, we give an explicit description of the period matrices of abelian threefolds as those that appear in \Cref{mainthmzp}. Following this, in \Cref{section:backgroundgfuns} we give a brief account of the process of associating a family of G-functions to a pair $(X\rightarrow S, s_0\in S(\bar{\Q}))$, consisting of a $1$-parameter family of abelian threefolds $X\rightarrow S$ and a distinguished point $s_0$ on the base curve.

In \Cref{section:relations} we describe the construction of relations among the values of these G-functions at points $s$ as in \Cref{mainthmzp}. This is handled on a ``case-by-case'' basis depending on the type of the place $v$ with respect to which $s$ is $v$-adically close to $s_0$. The main part of our exposition ends with \Cref{section:heightbounds} where we combine our relations with the Lang-Trotter conjecture for pairs of elliptic curves to establish the height bounds we need. \Cref{section:appendixmathematica} includes a series of codes written in Wolfram Mathematica that were used to establish the non-triviality of the relations constructed in \Cref{section:relations}.

\subsection{Notation}
Fix a number field $K$. We write $\Sigma_K$ for the set of places of $K$ and $\Sigma_{K,\infty}$ (resp. $\Sigma_{K,f}$) for the infinite (resp. finite) places of $K$. 

Fix $v\in\Sigma_{K}$. We write $K_v$ for the completion of $K$ at $v$. For any variety $X$ over $K$, we write $X_v:=X\times_K K_v$ and $X^{an}_v$ for its analytification. We write $\C_v$ for either $\C$ or $\C_p$ depending on whether $v\in\Sigma_{K,\infty}$ (real or complex) or $v\in\Sigma_{K,f}$. We write $\iota_v:K\hookrightarrow \C_v$ for the associated embedding. If $v\in\Sigma_{K,f}$, we write $k_v$ for the residue field of $K_v$, we write $W(k_v)$ for the ring of Witt vectors over $k_v$, and we write $K_{v,0}$ for the fraction field of $W(k_v)$.

Fix an abelian variety $A$ defined over $K$. We denote by $\End_{\bar{\Q}}(A)$ its ring of endomorphisms defined over $\bar{\Q}$ and define $\End^{0}_{\bar{\Q}}(A):=\End_{\bar{\Q}}(A)\otimes\Q$. If $v$ is a finite place of good reduction for $A$, we write $\tilde{A}_v$ for the reduction of $A_v$ modulo $v$. 

If \[y(x)=\Sum{n=0}{\infty}a_n x^n\in K[[x]]\] is a power series, we write \[\iota_v(y(x)):=\Sum{n=0}{\infty}\iota_v(a_n)x^n\] for the corresponding power series in $\C_v[[x]]$. For a family $\mathcal{Y}$ of power series $y_1\ld y_N\in K[[x]]$, we write $R_v(y_{j})$ for the $v$-adic radius of convergence of $y_j$ and $R_{v}(\mathcal{Y}):=\min R_v(y_j)$.\\

\textbf{Acknowledgments:} For part of this work, CD was supported by a grant from the School of Mathematics at the Institute for Advanced Studies, Princeton. CD acknowledges the use of ChatGPT for conceptualisation. GP thanks Gal Binyamini, Nick Katz, Dmitry Novikov, and Jonathan Pila for their questions following a lecture on this work, which led to several improvements in the exposition, including the discussion in \Cref{section:langtrotternumberfield}. Throughout this work GP was supported by the Minerva Research Foundation Member Fund while in residence at the Institute for Advanced Study for the academic year $2025$-$26$.

\section{Basics on period matrices}\label{section:periodmatrices}

The main objects of our study are period matrices and relations among their entries. We begin with a brief summary of these objects and record a convenient description of period matrices for the abelian varieties appearing in \Cref{mainthmzp}.

\subsection{Notation on period matrices}\label{section:notationperiodmatrices}
Throughout this subsection we fix an abelian variety $A$ defined over a number field $K$. For simplicity we assume that $A$ has everywhere semi-stable reduction over $K$ and write $\Sigma_{K, \good}(A)$ for the set of places of $K$ over which $A$ has good reduction.

Given a place $v \in \Sigma_{K, \good}(A) \cup \Sigma_{K,\infty}$ we set
$$
H_{v}^{1}(A)= \begin{cases}H^{1}\left(A_{v}^{a n}, \mathbb{Q}\right) & \text { if } v \in \Sigma_{K, \infty} \text { and } \\ H_{\crys}^{1}\left(\widetilde{A}_{v} / W\left(k_{v}\right)\right) \otimes K_{v, 0} & \text { if } v \in \Sigma_{K, f} ,\end{cases}
$$
where the first cohomology group is Betti cohomology and the latter is crystalline cohomology. 

With $A$ as above we have canonical comparison isomorphisms
$$
\rho_{v}(A): H_{d R}^{1}(A / K) \otimes\C \rightarrow H_{v}^{1}(A) \otimes \C,
$$
for $v \in \Sigma_{K, \infty}$ due to Grothendieck. For $v \in \Sigma_{K,\good}(A)$ we similarly have 
$$
\rho_{v}(A): H_{d R}^{1}(A / K) \otimes K_{v} \rightarrow H_{v}^{1}(A) \otimes K
$$
due to Berthelot-Ogus \cite{bertogus}.

Upon choosing bases of $H_{dR}^{1}(A / K)$ and $H_{v}^{1}(A)$ we record these isomorphisms via matrices. These bases will be implicit in our exposition, so we denote these matrices for simplicity by ``$\Pi_{v}(A)$''.

For simplicity, we work with bases of $H_{d R}^{1}$ that satisfy the following:
\begin{definition}\label{defnhodgebasis}
	Let $X$ be a principally polarized $g$-dimensional abelian variety over a number field $K$.  We call an ordered basis \begin{center}
		$\Gamma_{dR}(X):=\{\omega_1, \ldots,\omega_g,\eta_1,\ldots,\eta_g\}$
	\end{center} of $H^1_{dR}(X/K)$ a \textbf{Hodge basis} (of $X$) if the following are true:
\begin{enumerate}
	\item $\omega_1,\ldots,\omega_g$ are a basis of the first part of the filtration $F^1=e^{*}\Omega_{X/K}\subset H^1_{dR}(X/K)$, and 
	
	\item $\Gamma_{dR}(X)$ is a symplectic basis, meaning that $\langle \omega_i,\eta_{j}\rangle=\delta_{i,j}$ and $\langle \omega_i,\omega_j\rangle=\langle \eta_i,\eta_j\rangle=0$, with respect to the Riemann bilinear form.
\end{enumerate}\end{definition}

\subsection{Period matrices for certain abelian threefolds}\label{section:periodmatlemmas}

The period matrices we are most interested in are those associated to the abelian threefolds corresponding to the points in \Cref{mainthmzp}. We record here a convenient description of these matrices.

We begin with the following elementary:
\begin{lemma}\label{lemmaperiods1}
Let $E_{1},$ $E_{2}$, and $E_{3}$ be elliptic curves defined over a number field $K$ and let $A^{\prime}=E_{1} \times_K E_{2} \times_K E_{3}$. There exists a Hodge basis of $A^{\prime}$ such that for all $v \in \Sigma_{K,\good}\left(A^{\prime}\right) \cup \Sigma_{K, \infty}$ there exists a symplectic basis of $H_{v}^{1}\left(A^{\prime}\right)$ with respect to which
$$
\Pi_{v}\left(A^{\prime}\right)=J \left(\begin{array}{lll}
	\Pi_{v}\left(E_{1}\right) & 0& 0\\
	0& \Pi_{v}\left(E_{2}\right) &0 \\
	0& 0& \Pi_{v}\left(E_{3}\right)
\end{array}\right) J^{-1},
$$
where $J \in\GL_6(\Z)$ has entries in $\{0,1\}$.

Moreover, if some $E_{j}$ is CM and $v$ does not divide the discriminant of $\End^{0}_{\bar{\Q}}\left(E_{j}\right)$ the above bases may be chosen so that there exists $\varpi_{v} \in \mathbb{C}_{v}$ such that
$$
\Pi_{v}\left(E_{j}\right)= \begin{cases}\left(\begin{array}{cc}
		\varpi_{v} & 0 \\
		0 & \varpi_{v}^{-1}
	\end{array}\right), & \text { if } v \in \Sigma_{K, f} \\
	\left(\begin{array}{cc}
		\frac{\varpi_{v}}{2 \pi {i}} & 0 \\
		0 & \varpi_{v}^{-1}
	\end{array}\right), & \text { if } v \in \Sigma_{K, \infty}.\end{cases}
$$
\end{lemma}
\begin{proof}This follows by the same arguments as those in Remark $2.2$ and Lemma $2.5$ of \cite{papaspadicpart1}.
	
We note that the matrix $J$ corresponds to the change of basis matrix between ordered bases of the form $\left\{\omega_{1}, \eta_{1}, \omega_{2}, \eta_{2}, \omega_{3}, \eta_{3}\right\}$ and $\left\{\omega_{1}, \omega_{2}, \omega_{3}, \eta_{1}, \eta_{2}, \eta_{3}\right\}$. Here $\left\{\omega_{i}, \eta_{i}\right\}$ stands for a symplectic basis of $H_{d R}^{1}\left(E_{i} / K\right)$. In more detail, we record here that
$$J=\begin{pmatrix}
	1 & 0 & 0 & 0 & 0 & 0 \\
	0 & 0 & 0 & 1 & 0 & 0 \\
	0 & 1 & 0 & 0 & 0 & 0 \\
	0 & 0 & 0 & 0 & 1 & 0 \\
	0 & 0 & 1 & 0 & 0 & 0 \\
	0 & 0 & 0 & 0 & 0 & 1
\end{pmatrix}.$$

Finally, for the description of $\Pi_v(E_j)$ in the ``moreover part'' we refer the reader to Lemma $2.6$ of \cite{papaspadicpart2} and its proof.\end{proof}

Next we record the effect of isogenies in the above setting. Given an isogeny $\phi: A \rightarrow A^{\prime}$ over $K$, we write $[\phi]_{d R}$ for the matrix associated to the pullback of $\phi$ on $H_{d R}^{1}$ and implicit choices of Hodge bases. Similarly given $v \in \Sigma_{K, \good}(A) \cup \Sigma_{K, \infty}$ we write $[\phi]_{v}$ for the matrix associated to the pullback of $\phi$ on the respective $v$-adic cohomology $H^1_v$ and an implicit choice of bases.

\begin{lemma}\label{lemmaperiods2}
	Let $\phi: A \rightarrow A^{\prime}=E_{1} \times_K E_{2} \times_K E_{3}$ be an isogeny defined over $K$ and fix $v \in \Sigma_{K, \good}(A) \cup \Sigma_{K,\infty}$.
	\begin{enumerate}
		\item For any choice of Hodge bases and for any choice of bases of $v$-adic cohomology groups we have
		$$
		[\phi]_{d R} \cdot\Pi_{v}(A)=\Pi_{v}\left(A^{\prime}\right) \cdot[\phi]_{v},
		$$
		where $[\phi]_{d R}$ is lower triangular.
		
		\item Given a Hodge basis of $A^{\prime}$ we may choose a Hodge basis of $A$ such that
		$$
		[\phi]_{d R}=\begin{pmatrix}
			I_{3} & 0 \\
			B & C
		\end{pmatrix},
		$$where $B \in M_{3 \times 3}(K)$ and $C \in G L_{3}(K)$.\end{enumerate}\end{lemma} 
\begin{proof}
	This follows from the same arguments as those that appear in $\S 2.2.1$ of \cite{papaspadicpart1} .
\end{proof}

Combining the above lemmas yields the following:
\begin{cor}\label{periodsdescription}
	For $A$ and $A^{\prime}$ as in \Cref{lemmaperiods2} there exist Hodge bases of $A$ and $A^{\prime}$ such that
	\[
	\begin{pmatrix}	I_{3} & 0\\	B & C\end{pmatrix} \Pi_{v}(A)=J \begin{pmatrix}\Pi_{v}\left(E_{1}\right) & 0 & 0 \\
		0 & \Pi_{v}\left(E_{2}\right) & 0 \\0 & 0 & \Pi_{v}\left(E_{3}\right)\end{pmatrix} J^{-1} [\phi]_{v},
		\]
	where $C \in G L_{3}(K)$.
\end{cor} 
\section{Background on G-functions}\label{section:backgroundgfuns}

We give a summarized version of the connection between relative periods and G-functions. We suppress several technical details and work in a slightly simplified setting. These technical details, while necessary to establish the height bounds we want in full generality, would significantly increase the length of our exposition. For further details, tailored to the setting we follow here, we refer the reader to $\S 3$ of \cite{papaspadicpart1} and \cite{daworrpap}.

\subsection{Working assumptions}\label{section:backgroundassumptions}

Fix a family of principally polarized abelian threefolds $f:X\rightarrow S$ over a curve defined over $\bar{\Q}$. We assume that the curve $S$ is irreducible and smooth. We also fix a point $s_0\in S(\bar{\Q})$ which will serve as the ``center'' of our G-functions and write $X_0:=X_{s_0}$ for the fiber of our family over $s_0$. Let also $K$ be a number field over which both $f:X\rightarrow S$ and $s_0$ are defined.

To establish the height bounds we want, we are allowed to base change the geometric picture we start with, i.e. $f:X\rightarrow S$, with either a finite cover or open subset of the base curve $S$ or by a finite extension of $K$. See, for example, the exposition in $\S 2.2$ of \cite{papaszpy1}. 

After such a base change, thanks to work of the first named author with M. Orr in \cite{daworr4}, we may assume that our base curve is equipped with a local parameter $x\in K(S)$ at $s_0$ satisfying:
\begin{enumerate}
	\item $x$ has only simple zeroes, all defined over $K$, and
	\item the morphism $x:S\rightarrow \mathbb{P}^1$ is Galois.
\end{enumerate}

For more details, see Lemma $5.1$ and the discussion in $\S 5.1$ of \cite{daworr4}. We will refer to such an $x$ as a ``good local parameter of $S$ at $s_0$'' from now on.

Similarly, we may assume that the ``central fiber'' $X_0$ has everywhere semi-stable reduction over the field $K$, thanks to Grothendieck's semi-stable reduction Theorem, see, for example, $\S 7.4$ in \cite{neron}. In the same spirit, we assume that all of the endomorphisms of $X_0$ are defined over $K$, and so too are any isogenies that appear henceforth.

\subsection{G-functions associated to a family and a center}\label{section:gfunsassociated}

We consider a family $f:X\rightarrow S$ of principally polarized abelian threefolds defined over a number field $K$ together with a point $s_0\in S(K)$ satisfying the assumptions of \Cref{section:backgroundassumptions}. Suppose that $x$ is a good local parameter of $S$ at $x_0$.

Associated to the above data we have the differential module with connection given by $\mathcal{M}:=H^1_{dR}(X/S)$, the relative de Rham cohomology, and $\nabla$, its Gauss-Manin connection. Working in a sufficiently small affine open neighborhood $U$ of $s_0$ we may find a (full) basis of global sections of $\mathcal{M}(U)$, say $\beta$. The action of $\nabla$ with respect to $\beta$ encodes the differential module $\mathcal{M}(U)$ as a differential system over $\bar{\Q}(x)$ via the local parameter $x$.

By seminal work of Y. Andr\'e, see \cite{andre1989g}, this yields a matrix of G-functions $Y_{G,\beta,s_0}$ centered at $s_0$ which is a formal solution to this system, normalized by the condition $Y_{G,\beta,s_0}(x(s_0))=Y_{G,\beta,s_0}(0)=I_{6}$ (the identity matrix).

In fact, we require one such matrix $Y_{G,\beta,P}$ for all $P$ in the set of zeroes of the local parameter $x$. We suppress this technical issue for the sake of expositional simplicity, see \cite{daworrpap} for more details. To establish the height bounds in \Cref{heightboundintro}, the crucial input is given by the relations constructed in the next section and their properties, i.e. their ``non-triviality''. For further details, see the proof of the height bounds in \cite{papaspadicpart1}.

\subsubsection{G-functions and periods}\label{section:gfunsandperiodsbackground}

As above, we work in the simplified setting where $s_0$ is the only root of $x$. For notational simplicity we also write $Y_{G,\beta}:=Y_{G,\beta,s_0}$. Given a place $v\in \Sigma_K$ we write $\Delta_v$ for a sufficiently small analytic disk embedded in $S^{an}_v$ and centered at $s_0$. 

If $v\in \Sigma_{K}$ is a place over which the central fiber $X_0$ does not have bad reduction, there is a natural connection between the ``$v$-adic value'' of $\iota_v(Y_{G,\beta}(x(s)))$ and the ``$v$-adic periods'' of the fiber $X_s$ as defined in \Cref{section:periodmatrices}. Upon choosing $\Delta_v$ to be small enough we may identify the $v$-adic cohomology group $H^1_v(X_0)$ with horizontal sections of $\mathcal{M}^{\rm an}|_{\Delta_v}$. Then, for all $s\in \Delta_v$, we have \begin{equation}\label{eq:gfunsandperiods}
	\Pi_v(X_s)=\iota_v(Y_{G,\beta}(x(s)))\cdot \Pi_v(X_0).
\end{equation}

In the archimedean setting this follows from the relative version of the de Rham-Betti comparison isomorphism and was first noticed by Y. Andr\'e in \cite{andre1989g}. For finite places a similar description holds when one replaces Betti cohomology by crystalline cohomology. This is due to work of Berthelot-Ogus \cite{bertogus} and Ogus \cite{ogussolo}. For places $v$ over primes unramified in the extension $K/\Q$ this observation is again due to Y. Andr\'e, see \cite{andremots}. The description in \eqref{eq:gfunsandperiods} for the remaining finite places follows by the same argument as in Andr\'e's \cite{andremots} but this time using Ogus' overconvergent $F$-isocrystals from \cite{ogussolo}. For more details on this we refer the reader to \cite{daworrpap,papaspadicpart1}.

\subsubsection{Choice of bases and splittings}\label{section:derhambasis}

Choosing the above basis $\beta$ appropriately simplifies our exposition. In the setting of \Cref{mainthmzp}, we choose the fiber at $s_0$ of the basis to be compatible with \Cref{periodsdescription}.

In order to be able to talk about period matrices we also choose bases for the $v$-adic cohomology $H^1_v(X_0)$, which we identify with horizontal sections of the connection in a sufficiently small disk around $s_0$. On this side of the comparison isomorphism we will only need to choose such bases that are symplectic with respect to the Riemann bilinear form.

Assume that $X_0$ is isogenous to a threefold of the form $X_0'=E_1\times_K E_2\times_K E_3$. The one place where we will make additional assumptions about the bases of $H^1_{\crys}$ is for the bases chosen for the $H^1_{\crys}(E_i/K_v)$. These bases, after reordering, induce the basis of $H^1_v(X_0')$ alluded to in \Cref{lemmaperiods1}. In the case where $E_i$ has good ordinary reduction at $v$ we choose this basis as described in $\S 2.1.1$ of \cite{papaspadicpart2}. 

\subsubsection{$v$-adic proximity}\label{section:vadicprox}

Fix $s\in S(\bar{\Q})$ and $w\in \Sigma_{K(s)}$ with $w|v$ for some $v\in \Sigma_{K}$. We set $r_v(Y_{G,\beta}):=\min\{1,R_v(Y_{G,\beta})\}$. 

For finite places $w$ as above we have two possible notions of what it means for $s$ to be close to $s_0$ with respect to $w$. The first is that $|x(s)|_w\leq r_w(Y_{G,\beta})$, or that $x(s)$ is within a $w$-adic disk of convergence of the family $Y_{G,\beta}$. The second is that $X_0$ and $X_s$ have the same reduction modulo $w$, or in other words that $s$ and $s_0$ are $w$-adically close with respect to the $w$-adic geometry of $S$. 

The two notions of proximity coincide for all but finitely many $v\in\Sigma_{K}$, see the discussion in Chapter $X$, $\S 3.1$ of \cite{andre1989g}. Following the ideas of Andr\'e in loc. cit., there are two ways to get around this. One is to introduce an auxiliary G-function as is done in \cite{daworr4,daworrpap}. The other is to alter the basis $\beta$ by multiplying its entries with an appropriate rational function on an affine open neighborhood of $s_0$. This is the direction considered in \cite{papaspadicpart1} where we show that this choice may be done in a way that the new basis $\beta'$ is still symplectic if the original basis was so. 

\begin{lemma}\label{associatedGfunslemma}
	There exists a symplectic basis $\beta$ of sections of $H^1_{dR}(X/S)$, defined on a sufficiently small affine open neighborhood of $s_0$, for which the two notions of $v$-adic proximity coincide for all finite places $v$. In other words, $|x(s)|_w\leq r_w(Y_{G,\beta})$ if and only if $X_s$ and $X_0$ are equal modulo $w$.	
	\end{lemma}

\begin{definition}\label{defnassociatedGfuns}
	We write $Y_{G}:= Y_{G,\beta}$ for a basis of sections that satisfies \Cref{associatedGfunslemma} and call this a \textbf{matrix of G-functions associated to the pair} $(X\rightarrow S, s_0)$.
\end{definition}

\subsection{Trivial relations}\label{section:trivialrelations}
One consequence of imposing in \Cref{associatedGfunslemma}, and hence \Cref{defnassociatedGfuns}, that the basis $\beta$ be symplectic is that it yields a simple description of the ``trivial relations'' of our family of G-functions in the vocabulary of \cite[VII.5]{andre1989g}.

\begin{prop}\label{proptrivialrelations}
	Let $Y_G\in M_6(\bar{\Q}[[x]])$ be a matrix of G-functions associated to the pair $(X\rightarrow S, s_0)$. Assume that the image of the induced morphism $\iota:S\rightarrow \mathcal{A}_3$ is a curve not contained in a proper special subvariety of $\mathcal{A}_3$.
	
	Then the $\bar{\Q}(x)$-Zariski closure $Y_G^{\zar}\subset (M_6)_{\bar{\Q}(x)}$ is cut out by the ideal $I(\SP_6)$ that defines the group $\SP_6$ as a subvariety of the variety $M_6$ of $6\times 6$ matrices.
\end{prop}
\begin{proof}
This follows by J. Ayoub's affirmative answer to the relative version of the Kontsevich-Zagier period conjecture, see \cite{ayoub}. For further details, we refer the reader to $\S 4.5.4$ of \cite{daworrpap}.
\end{proof}
\begin{remark}
	Strictly speaking, we require a statement closer to $3.12$ in \cite{papaspadicpart1}. The statement there accounts for the possibility that the local parameter $x$ can have many zeroes. We suppress this subtlety for expositional simplicity. This is especially because the more involved case, that is when $x$ has multiple zeroes, reduces to the case considered here, at least for the purpose of establishing the ``non-triviality'' of the relations.
\end{remark}

\section{Relations among values of G-functions}\label{section:relations}

In this section, we construct relations among the $v$-adic values of our G-functions at points $s$ as appearing in \Cref{mainthmzp}. We do this on a case-by-case basis depending on the type of the place $v$.
\subsection{Setting}\label{section:relationssetting}
We fix from now on a family of principally polarized abelian threefolds $f: A \rightarrow S$ defined over a number field $K$ together with a point $s_{0} \in S(K)$ for which $A_{0}:=A_{s_{0}}$ is isogenous to a threefold of the form $A_{0}^{\prime}:=E_{1} \times_K E_{2} \times_K E_{3}$. We also work under the assumptions that $E_{1}$ is CM and $\left(E_{2}, E_{3}\right)$ is a Hodge generic point in $Y(1)^{2}$. We shall also write $\phi: A_{0} \rightarrow A_{0}^{\prime}$ for the aforementioned isogeny.

We assume that our family is as in the ``simplified setting'' described in \Cref{section:backgroundgfuns}. In this sense, from now on, we denote by $Y_{G} \in \SP_6(K[[x]])$ an associated matrix of $G$-functions centered at $s_{0}$.

From now on we write $\Sigma_{K,G}$ for the union of the set of archimedean places $\Sigma_{K, \infty}$ and the set $\Sigma_{K,\good}(A_0)$ of places of good reduction of $A_{0}$. Given a finite extension $L / K$ we also let
$$
\Sigma_{L, G}:=\left\{v \in \Sigma_{L}: \exists w \in \Sigma_{K, G} \text { with } v \mid w\right\}.
$$

Throughout this section we also fix $s \in S(\bar{\Q})$ as in \Cref{mainthmzp}. In particular there exists an isogeny
$$
\psi: A_{s} \longrightarrow A_{s}^{\prime}=E_{s}^{3},
$$
for some elliptic curve $E_s$, defined over $\bar{\Q}$. In fact, it is classical that there exists a finite extension $L_{s} / K(s)$ over which $\psi$ is defined and that, see \cite{silverberg} for example, $[L_s: K(s)]$ may be bounded by an absolute constant.

With the above in mind we assume from now on that $L_{s}=K(s)$. We also fix from now on $v \in \Sigma_{K(s), G}$ a place with respect to which $s$ is $v$-adically close to $s_{0}$.

To simplify our exposition, we also work under the following assumptions.
\begin{enumerate}
	\item $F_0:=\End^{0}_{\bar{\Q}}\left(E_{1}\right)=\End^{0}_{K}\left(E_{1}\right)$ and 
	\item $A_{0}$, and hence $E_{i}$ for $1 \leq i \leq 3$, have everywhere semi-stable reduction.
\end{enumerate}

The nature of the points we are trying to count, as well as the assumptions on $A_{0}$, allow us to give a concrete description of the possible types of places of $\Sigma_{K(s)}$ with respect to which $s$ is close to $s_0$.

\begin{lemma}\label{lemmatypesofproximity}
 Let $v \in \Sigma_{K(s)}$ with $v \mid w \in \Sigma_{K}$ be a place for which $s$ is $v$-adically close to $s_{0}$. Then 
\begin{enumerate}
	\item $w \in \Sigma_{K,\infty}$, or
	\item $w\in \Sigma_{K,f}$ and $\tilde{E}_{i, w}$ is an ordinary elliptic curve for $1 \leq i \leq 3$, or
	\item $w\in \Sigma_{K,f}$ and $\tilde{E}_{i, w}$ is a supersingular elliptic curve for $1 \leq i \leq 3$.\end{enumerate}\end{lemma}

\begin{proof}
	Assume $w \notin \Sigma_{K,\infty}$. By our conventions we have that $\tilde{A}_{s, v}=\tilde{A}_{0, v}$. In particular, $\tilde{E}_{s,v}^{3}$ is isogenous to $\tilde{E}_{1, v} \times_{k(v)} \tilde{E}_{2, v} \times_{k(v)} \tilde{E}_{3, v}$ and thus $\tilde{E}_{i, v}$ is isogenous to $E_{s, v}$ for $1 \leq i \leq 3$. Since $A_{0}$ has everywhere semi-stable reduction and $E_{1}$ is CM, $E_{1}$ will have everywhere good reduction, see \cite{serretate}. Finally, all $\tilde{E}_{i, v}$ are isogenous and will thus have the same type of reduction.
\end{proof} 

\begin{remark}
	For $s\in S(\bar{\Q})$ as above, we write $Y_G(s):=Y_G(x(s))$ for the remainder of this section.
\end{remark}

\subsection{The relation generator}\label{section:relgenerator}
To make our notation a bit more compact we will write
\begin{equation}
	\diag\Pi_v(E_i)=\begin{pmatrix}\Pi_v(E_1)&0&0\\0&\Pi_v(E_2)&0\\0&0&\Pi_v(E_3)	\end{pmatrix}, 
\end{equation}to denote the obvious block-diagonal matrices. When $E=E_1=E_2=E_3$ we write simply $\diag\Pi_v(E)$. Also, following usual notational conventions, given an isogeny $\phi:A\rightarrow B$ of abelian varieties we write $\phi^{\vee}$ for its dual.

The main equation we use to generate relations is given by the following:
\begin{lemma}\label{lemmarelgenerator}
	Let $\left(f: A \rightarrow S, s_{0}\right)$, $s \in S(\bar{\Q}), v \in \Sigma_{K(s),G}$ and $\psi: A_{s} \rightarrow A_{s}^{\prime}=E_s^3$ be as above. Then, 
	\begin{equation}\label{eq:relgenerator}
	\iota_{v}\left(J^{-1}[\psi]_{d R} Y_{G}(s)\left[\phi^{\vee}\right]_{dR} J\right)=\diag\Pi_{v}\left(E_{s}\right) \Theta_{v}  \diag \Pi_{v}\left(E_{i}\right)^{-1},
	\end{equation}where $\Theta_{v}=J^{-1} [\psi]_{v}  [\phi^{\vee}]_{v}  J$.
\end{lemma}
\begin{proof}
	 This follows by the same argument as Lemma $4.1$ in \cite{papaspadicpart1}. In short, from \Cref{lemmaperiods2} we have
	$$
	[\psi]_{dR} \Pi_{v}\left(A_{s}\right)=\Pi_{v}\left(A_{s}^{\prime}\right) [\psi]_{v} .
	$$
	
	Noting that $\Pi_{v}\left(A_{s}\right)=\iota_{v}\left(Y_{G}(s)\right)  \Pi_{v}\left(A_{0}\right)$ and applying \Cref{lemmaperiods1} once for $A_{s}^{\prime}$ and once for $A_{0}^{\prime}$, after using  \Cref{periodsdescription} for $A_0$, we get
	\begin{equation}
		\iota_{v}\left(J^{-1} [\psi]_{d R} Y_{G}(s) [\phi]_{d R}^{-1} J\right)=\diag\Pi_{v}(E_s)  \Theta_{v}  \diag \Pi_{v}\left(E_{i}\right)^{-1},
	\end{equation}
	where $\Theta_{v}:=J^{-1}[\psi]_{v}[\phi]_{v}^{-1} J$. Noting that $[\phi^{\vee}]_{*}=\operatorname{deg}( \phi) \cdot[\phi]_{*}^{-1}$, for $* \in\{d R, v\}$, we are done.\end{proof}

The matrix $\Theta_{v}$ captures further geometric information. We encode this in the following:
\begin{lemma}\label{matrixtheta}
	Let $v \in \Sigma_{K, G}$ and $s \in S(K)$ be as in \Cref{section:relationssetting} and $\Theta_v$ be the matrix defined in \Cref{lemmarelgenerator}.
	
1. If $v \in \Sigma_{K(s), \infty}$ then $\Theta_{v} \in G L_{6}(\mathbb{Q})$.

2. If $v \in \Sigma_{K(s),f}$, then for $1 \leq i, j \leq 3$ there exists a morphism $\tilde{\theta}_{i, j}: \tilde{E}_{j, v} \rightarrow \tilde{E}_{s, v}$ such that $\Theta_{v}=\left(\theta_{i, j}\right)$ where $\theta_{i, j}$ denotes the matrix associated to the pullback of $\tilde{\theta}_{i j}$ in crystalline cohomology. 

3. If $v \in \Sigma_{K(s),f}$ and $v$ is ordinary for the $E_{i}$ then for $1 \leq i, j \leq 3$ there exist $\xi_{i j}$, $\zeta_{ij} \in \bar{\Q}$ such that $\theta_{i j}=\begin{pmatrix}\xi_{i j} & 0 \\ 0 & \zeta_{i j}\end{pmatrix}$.
\end{lemma} 

\begin{proof}The first part follows trivially from functoriality in the de Rham Betti comparison isomorphism.
	
	The second and third part follow from our choice of bases for $H^{1}_{\crys}$. See Lemma $4.3$ of \cite{papaspadicpart1} and Lemma $2.4$ of \cite{papaspadicpart2} respectively for more details on those.
\end{proof} 
From now on we set
\begin{equation}\label{eq:thematrixz}
	Z=\left(Z_{i, j}\right):=\iota_{v}\left(J^{-1}[\psi]_{d R} Y_{G}(s)\left[\phi^{\vee}\right]_{d R} J\right).\end{equation}

\subsection{Archimedean and supersingular places}\label{section:relarchandsuper}

We start with the archimedean places and then proceed to the places for which the $E_i$ have supersingular reduction. The construction of relations for these places share significant similarity. We emphasize that these relations will depend on the place $v$.
\begin{prop}\label{proparchimedean}
Assume $v \in \Sigma_{K(s), \infty}$. Then there exists $R_{s, v} \in K(s)\left[X_{i j}: 1 \leq i, j \leq 6\right]$ such that
	\begin{enumerate}
		\item $\iota_v\left(R_{s, v}\left(Y_{G}(s)\right)\right)=0$,
		\item $R_{s, v}$ is homogeneous of degree $2$, and
		\item $R_{s,v} \notin I\left(\SP_6\right)$.
	\end{enumerate}
\end{prop}
\begin{proof}
	We write the matrix $Z$ of \eqref{eq:thematrixz} as $Z=\left(F_{i j}\right)$ where $F_{i j}$ are $2 \times 2$ matrices.
	
	From \Cref{lemmarelgenerator} and \Cref{matrixtheta} we get	
	$$
	F_{11}=\Pi_{v}\left(E_{s}\right) \cdot \theta_{11} \cdot \Pi_{v}\left(E_{1}\right)^{-1}.
	$$
	Noting that $\det \Pi_{v}\left(E\right)=(2 \pi i)^{-1}$ for the period matrices of all elliptic curves that appear here and taking determinants in the above, we get
	$$
	\det F_{11}=\det \theta_{11}=: d_{v} .
	$$
	
	By \Cref{matrixtheta} we know $d_{v} \in \mathbb{Q}$. Consider the polynomial
	$$g(\underline{X})=X_{11}X_{44}-X_{14}X_{41}+X_{21}X_{54}-X_{24}X_{51}+X_{31}X_{64}-X_{34}X_{61}$$ 
	and note that $g(\underline{X})-1 \in I\left(\SP_6\right)$. In particular, from \Cref{proptrivialrelations} we get $g\left(Y_{G}(x)\right)=1$ on the level of power series.
	
	All in all, we get the relation
\begin{equation}\label{eq:archrelation}
	\det F_{11}-d_{v} \cdot \iota_{v}\left(g\left(Y_{G}(s)\right)\right)=0 .
\end{equation}
	Let us write $R_{s, v}$ for the polynomial obtained from (\ref{eq:archrelation}) by replacing the $ij$ entry of $Y_G(s)$ by the variable $X_{ij}$. Then $R_{s,v }$ satisfies all but the last property we want.
	
	Assume to the contrary that $R_{s,v} \in I\left(\SP_6\right)$.	The code in \Cref{section:apparchimedean} outputs the coefficients and monomials of the remainder of the division of $R_{s,v}$ by a Gr\"obner basis of $I\left(\SP_6\right)$. Since $R_{s,v} \in I\left(\SP_6\right)$ the coefficients of all the monomials of this remainder must be zero. We write $c\left(\underline{X}^{\underline{\alpha}}\right)$ for the coefficient corresponding to the monomial $\underline{X}^{\underline{\alpha}}$.
	
	Since all bases of $H_{d R}^{1}$ are Hodge bases by assumption, the matrix $[\psi]_{d R}$ is lower triangular of the form
	$$
	\begin{pmatrix}
		A_{s} & 0 \\
		B_{s} & C_{s}
	\end{pmatrix},
	$$
	where $A_{s}$, $B_{s}$, and $C_{s}$ are $3 \times 3$ matrices. Since $\psi$ is an isogeny we also get $A_{s}$, $C_{s} \in \GL_{3}(\bar{\Q})$. Setting $A_{s}=\left(d_{i j}\right)$ the code in \Cref{section:apparchimedean} outputs
	$$
	\begin{aligned}
		& c\left(X_{11} X_{23}\right)=d_{11} d_{32}-d_{12} d_{31}=0, \\
		& c\left(X_{11} X_{33}\right)=d_{11} d_{33}-d_{13} d_{31}=0, \text { and } \\
		& c\left(X_{21} X_{33}\right)=d_{12} d_{33}-d_{13} d_{32}=0 .
	\end{aligned}
	$$
	
	These would imply that the second column of the adjugate of $A_{s}$ is zero. This clearly contradicts $A_{s} \in \GL_{3}(\bar{\Q})$.\end{proof}

\begin{prop}\label{propsupersing}
	Assume $v \in \Sigma_{K(s),f}$ is a place of supersingular reduction of $E_{i}$ for $1 \leq i \leq 3$.
	
	There exists $R_{s, v} \in K(s)\left[X_{i j}: 1 \leq i, j \leq 6\right]$ such that
	\begin{enumerate}
		\item $\iota_v\left(R_{s, v}\left(Y_{G}(s)\right)\right)=0$,
		\item $R_{s, v}$ is homogeneous of degree $2$, and
		\item $R_{s, v} \notin I\left(\SP_6\right)$.
	\end{enumerate}
\end{prop} 
\begin{proof}We follow the notation in the proof of \Cref{proparchimedean} to rewrite \eqref{eq:relgenerator} as
\begin{equation}\label{eq:relgenrewritten}
	\left(F_{i j}\right)=\begin{pmatrix}\Pi_{v}\left(E_{s}\right) & 0 & 0 \\ 0 & \Pi_{v}\left(E_{s}\right) & 0 \\ 0 & 0 & \Pi_{v}\left(E_{s}\right)\end{pmatrix}\left(\theta_{i j}\right)\begin{pmatrix}\Pi_{v}\left(E_{1}\right) & 0 & 0 \\ 0 & \Pi_{v}\left(E_{2}\right)& 0 \\ 0 & 0 & \Pi_{v}\left(E_{3}\right)\end{pmatrix}^{-1},
\end{equation}
	where $F_{i j}$ and $\theta_{i j}$ for $1 \leq i, j \leq 3$ denote $2 \times 2$ matrices. From \Cref{matrixtheta} we know that each $\theta_{i j}$ is the matrix associated to the pullback of a morphism $\widetilde{\theta}_{i j}: \widetilde{E}_{j, v} \rightarrow \widetilde{E}_{s,v}$.
	
	We now consider the morphisms $\widetilde{\theta}_{11}: \widetilde{E}_{1, v} \rightarrow \widetilde{E}_{s,v}$ and $\widetilde{\theta}_{21}: \widetilde{E}_{1, v} \rightarrow \widetilde{E}_{s,v}$. If $\widetilde{\theta}_{11}=0$, then arguing as in \Cref{proparchimedean} we get $\det F_{11}=0$. The polynomial $R_{s,v}$ corresponding to $\det F_{11}=0$ will satisfy all the properties we want with the possible exception of its non-triviality, i.e. $R_{s, v} \notin I\left(\SP_6\right)$. This follows by the same computation as in the previous proof. To see this note that the $d_v$ that appears in \eqref{eq:archrelation} does not appear in any of the coefficients of the remainder used in the previous proof.
	
	From now on, we assume $\widetilde{\theta}_{11}$ is an isogeny. We consider the endomorphism $\widetilde{\varphi}:=\widetilde{\theta}_{21} \circ \widetilde{\theta}_{11}^{\vee} \in \End(\widetilde{E}_{s,v})$, where $\widetilde{\theta}_{11}^{\vee}$ is the dual of $\widetilde{\theta}_{11}$.
	
	Arguing much as in the proof of Proposition $4.10$ in \cite{papaspadicpart1}, it is straightforward to see that the matrix corresponding to $\widetilde{\varphi}$ is $[\widetilde{\varphi}]=(\deg \widetilde{\theta}_{11}) \theta_{21} \theta_{11}^{-1}$. Furthermore, we have $e_{v}:=\operatorname{det}[\widetilde{\varphi}] \in \bar{\Q}$.
	
	By \eqref{eq:relgenrewritten} we get
	$$
	F_{21} F_{11}^{-1}=\Pi_{v}\left(E_{s}\right) \theta_{21} \theta_{11}^{-1} \Pi_{v}\left(E_{s}\right)^{-1}.
	$$
	We set $d_{v}:=\frac{e_{v}}{\deg \widetilde{\theta}_{11}} \in \overline{\mathbb{Q}}$ and take determinants in the above to get
	$$
	\det F_{21}-d_{v} \det F_{11}=0.
	$$
	
	This corresponds to a degree $2$ homogeneous polynomial which we denote again by $R_{s,v}$. This will satisfy all the properties we want by construction with the possible exception of the last one.
	
	As in the previous proof we assume $R_{s, v} \in I\left(\SP_6\right)$. We now use the code in \Cref{section:appsupersing} and similar notation as in the previous proof. From $c\left(X_{33} X_{51}\right)=0$, $c\left(X_{23} X_{61}\right)=0$, $c\left(X_{21} X_{43}\right)=0$, $c\left(X_{11} X_{53}\right)=0$, and $c\left(X_{11} X_{63}\right)=0$ we get
\begin{equation}\label{eq:supersingoutputcode}
	\begin{aligned}
		& d_{23} e_{22}=0, \text{ }d_{22} e_{23}=0, \text{ }d_{22} e_{21}=0, \\
		& d_{21} e_{22}=0, \text { and } d_{21} e_{23}=0,
	\end{aligned}\end{equation}
where here $C_{s}=\left(e_{i j}\right)$ and $A_{s}=\left(d_{i j}\right)$ in the notation of the previous proof.

	We work in cases. First assume $d_{21}=0$. We also have $d_{22} e_{23}=0$. If $\quad d_{21}=d_{22}=0$ then $d_{23} \neq 0$, since $A_{s} \in \GL_{3}(\bar{\Q})$. This forces $e_{22}=0$ from the above equations. At this point, looking at $c\left(X_{33} X_{41}\right)$, we also get $d_{23} e_{21}=0$, forcing $e_{21}=0$. Since $C_{s} \in \GL_{3}(\bar{\Q})$ we get $e_{23} \neq 0$. We reach a contradiction by looking at $c\left(X_{33} X_{61}\right)=0$ which gives $e_{23} d_{23}=0$.
	
	So if $d_{21}=0$ we must have $d_{22} \neq 0$ and $e_{21}=e_{23}=0$ from our first equations. Again $\det\left(C_{s}\right) \neq 0$ forces $e_{22} \neq 0$. Again from \eqref{eq:supersingoutputcode} we get $d_{23}=0$. Under these constraints $c\left(X_{23} X_{51}\right)=0$ gives $e_{22} d_{22}=0$, a contradiction.
	
	We must thus have $d_{21} \neq 0$ which forces $e_{22}=e_{23}=0$ in \eqref{eq:supersingoutputcode}. Again $\det\left(C_{s}\right) \neq 0$ implies $e_{21} \neq 0$. Applying this to \eqref{eq:supersingoutputcode} now forces $d_{2 2}=0$. The equation $c\left(X_{31} X_{43}\right)=0$ now gives $d_{23} e_{21}=0$, and thus $d_{23}=0$. Finally, $c\left(X_{33} X_{61}\right)=0$ under these constraints simplifies to $e_{21} d_{21}=0$, a contradiction.\end{proof}

\subsection{Places of ordinary reduction}\label{section:relordinary}
By \Cref{lemmatypesofproximity}, it remains to treat the places of ordinary reduction for the elliptic curves $E_{i}$, $1 \leq i \leq 3$. Here, a crucial observation is that the polynomial $R_{s,\ord}$ we construct does not depend on the place $v$. This is a new feature in the G-functions method that appears also in \cite{papaspadicpart1,papaspadicpart2,papaspadicpart3} for $\mathcal{A}_2$ and \cite{daworrpap} for $Y(1)^n$.
\begin{prop}\label{propordinary}
	 There exists a polynomial $R_{s,\ord} \in K(s)\left[X_{i j}: 1 \leq i, j \leq 6\right]$ such that
	\begin{enumerate}
		\item $\iota_{v}\left(R_{s, \ord}\left(Y_{G}(s)\right)\right)=0$ for all $v \in \sum_{K(s), f}$ that are ordinary for $E_{i}$ for $1 \leq i \leq 3$.
		\item $R_{s,\ord}$ is homogeneous of degree at most $3$, and
		\item $R_{s, \ord} \notin I\left(\SP_6\right)$.
	\end{enumerate}
\end{prop}
\begin{proof}
	We follow the notation of the last two proofs. From \eqref{eq:relgenrewritten} we get
	$$
	\begin{aligned}
		& F_{11}=\Pi_{v}\left(E_{s}\right) \cdot \theta_{11} \cdot \Pi_{v}\left(E_{1}\right)^{-1} \text { and } \\
		& F_{21}=\Pi_{v}\left(E_{s}\right) \cdot \theta_{21} \cdot \Pi_{v}\left(E_{1}\right)^{-1} .
	\end{aligned}
	$$
	
	Since $v$ is a place of ordinary reduction of the CM elliptic curve $E_{1}$ it follows, by Deuring's reduction theorem, that $v \not| \disc\left(\End_{\bar{\Q}}^{0}\left(E_{1}\right)\right)$. We may thus use \Cref{lemmaperiods1} to write $\Pi_{v}\left(E_{1}\right)=\begin{pmatrix}\varpi & 0 \\ 0 & \varpi^{-1}\end{pmatrix}$ for some $\varpi \in \mathbb{C}_{v}$. 
	
	Also from \Cref{matrixtheta} we may find $\xi_{i j}, \zeta_{i j} \in \mathbb{C}_{v}$ for which $\theta_{i j}=\begin{pmatrix}\xi_{i j} & 0 \\ 0 & \zeta_{i j}\end{pmatrix}$. Setting $\mu_{i 1}:=\xi_{i 1}\varpi^{-1}$ and $\lambda_{i 1}:=\zeta_{i 1} \varpi$, we may rewrite the above as
	\begin{equation}\label{eq:ordinaryproof1}
	\begin{aligned}
	& F_{11}=\Pi_{v}\left(E_{s}\right) \cdot\begin{pmatrix}	\mu_{11} & 0 \\0 & \lambda_{11}	\end{pmatrix} \text { and } \\
	& F_{21}=\Pi_{v}\left(E_{s}\right) \cdot\begin{pmatrix}\mu_{21} & 0 \\0 & \lambda_{21}	\end{pmatrix} .
\end{aligned}
	\end{equation}

	If $\tilde{\theta}_{11}$ is the zero morphism then $\mu_{11}=\lambda_{11}=0$ which forces $F_{11}=0$. Here we let $R_{s,1}$ be the polynomial in $K(s)\left[X_{i, j}\right]$ corresponding to the top left entry of $F_{1,1}$. By construction, this will be a degree 1 homogeneous polynomial that will satisfy the first property for all $v$ ordinary for which $\tilde{\theta}_{1,1}=0$.
	
	We also note here that $R_{s,1} \notin I\left(\SP_6\right)$. This may be seen for example by the fact that $R_{s,1}$ is linear and $I\left(\SP_6\right)$ has a minimal basis that comprises of polynomials of degree 2 .
	
	Now assume $\widetilde{\theta}_{11} \neq 0$ and is hence an isogeny of elliptic curves. From this we get that $\theta_{11}$, and thus also $F_{11}$, is invertible. From \eqref{eq:ordinaryproof1} we then get
	$$
	F_{11}^{-1} F_{21}=\begin{pmatrix}\mu_{21} \mu_{11}^{-1} & 0\\0 & \lambda_{21} \lambda_{11}^{-1}\end{pmatrix}.
	$$
Multiplying both sides of the above by $\det F_{11} \neq 0$ we get
\begin{equation}\label{eq:relationsordinaryhard}
	F_{11}^{\adj} \cdot F_{21}=\begin{pmatrix}\mu & 0 \\0 & \lambda	\end{pmatrix}\end{equation}
	for some $\mu, \lambda \in \mathbb{C}_{v}$.
	
	We set $\left(A_{i j}\right):=F_{11}^{\adj} \cdot F_{21}$. We then have the relations $A_{12}=A_{21}=0$. We set $R_{s, 2}$ to be the polynomial corresponding to $A_{12}$.
	
	At this point we let $R_{s,\ord}:=R_{s, 1} \cdot R_{s, 2}$. By our exposition it follows that $R_{s,\ord}$ satisfies all the properties we want with the only possible exception of $R_{s,\ord} \notin I\left(\SP_6\right)$. We assume this fails. Since trivially $R_{s, 1} \notin I\left(\SP_6\right)$ and this ideal is prime we must have $R_{s, 2} \in I\left(\SP_6\right)$. This leads to a contradiction working much as in the previous proofs.
	
	In more detail, we write $A_{s}=\left(d_{i j}\right)$ and $C_{s}=\left(e_{i j}\right)$ to be the elements of $\GL_{3}(\bar{\mathbb{Q}})$ that appear in the proof of \Cref{proparchimedean}. The code in \Cref{section:appordinary} outputs a list of monomials and coefficients for the remainder of $R_{s, 2}$ divided by a Gr\"obner basis of $I\left(\SP_6\right)$.
	
	The equations $c\left(X_{13} X_{63}\right)=c\left(X_{13} X_{43}\right)=c\left(X_{13} X_{53}\right)=0$ give $d_{31} e_{21}=d_{31} e_{22}=d_{31} e_{33}=0$. Since $\det\left(e_{i j}\right) \neq 0$ we get $d_{31}=0$. Similarly, the equations $c\left(X_{23} X_{63}\right)=c\left(X_{23} X_{43}\right)=c\left(X_{23} X_{53}\right)=0$ give $d_{32} e_{21}=d_{32} e_{22}=d_{32} e_{33}=0$. These give $d_{32}=0$. Finally, looking at the equations $c\left(X_{33} X_{63}\right)=c\left(X_{33} X_{43}\right)=c\left(X_{33} X_{53}\right)=0$ gives $d_{33} e_{21}=d_{33} e_{22}=d_{33} e_{33}=0$. This now forces $d_{33}=0$, and thus $\det\left(A_{s}\right)=0$ contradicting $A_{s} \in \GL_{3}(\bar{\Q})$.
\end{proof}
\section{Height bounds and applications}\label{section:heightbounds}

In this final section, we establish the height bounds alluded to in the introduction. A key additional input is a version of the Lang–Trotter conjecture for pairs of elliptic curves, which we recall before proceeding to the proof.
\subsection{Lang-Trotter for pairs}\label{section:langtrotter}
In \cite{langtrotter} S. Lang and H. Trotter proposed a probabilistic model for the distribution of the Frobenius traces $a_p(E)$ of an elliptic curve $E$ defined over $\Q$. Their model predicts that, for a fixed integer $t$, the quantity $$\pi_{E,t}(x):=\#\{p\leq x :a_p(E)=t \}$$will have asymptotic behavior of the form \begin{equation}\label{eq:onecurvelangtrott}
	\pi_{E,t}(x)\sim C_{E,t} \frac{\sqrt{x}}{\log x}.
\end{equation} Moreover the construction of their model is such that it gives an explicit description for the constant $C_{E,t}$.

In Remark $2$, page $37$ of \cite{langtrotter} Lang and Trotter, based on natural asymptotics related to the Sato-Tate conjecture, also make the following:
\begin{conj}[Simultaneous supersingular reduction for pairs \cite{langtrotter}]\label{langtrottersupersingularpairs} Let $\left(E_{1}, E_{2}\right)$ be a pair of elliptic curves defined over $\Q$. Assume that neither of these is CM and that they are not isogenous over $\bar{\Q}$.
	
There exists a non-negative constant $c_{E_1,E_2}$ such that
	$$
\pi_{E_1,E_2}(x):=\#\left\{p \leqslant x: E_1\text{ and }E_2\text{ have supersingular reduction modulo }p\right\} \sim c_{E_1,E_2} \log \log x,
	$$as $x\rightarrow \infty$.
\end{conj} 

\begin{remarks}$1$. \Cref{langtrottersupersingularpairs} is a special case of a ``Lang-Trotter conjecture for pairs of elliptic curves'' that recently appeared in \cite{akbaryparks}. The authors describe the natural analogue of the probabilistic model described by Lang-Trotter in \cite{langtrotter}, this time for pairs $(E_1,E_2)$ of elliptic curves that are as in \Cref{langtrottersupersingularpairs}. 
	
	Their model leads to a description of the asymptotic behavior of the quantity
	$$\pi_{E_1,E_2,t_1,t_2}(x):=\#\left\{p \leqslant x:p\nmid N_1N_2,\text{ } a_{p}\left(E_{1}\right)=t_1,\text{ }a_{p}\left(E_{2}\right)=t_2\right\}, $$
	where $t_1$ and $t_2$ are integers and $N_i$ stands for the conductor of the elliptic curve $E_i$. See Conjecture $1.2$ in \cite{akbaryparks}. In particular, they give conjectural descriptions for the constant $c_{E_1,E_2}$ of \Cref{langtrottersupersingularpairs}, i.e. the case $t_1=t_2=0$ in their notation.\\
	
	$2$. The original conjecture of Lang-Trotter remains widely open. The best known results in the case $t=0$ are due to N. Elkies, see \cite{elkiesLT}, who shows that $\pi_{E,0}(x)\ll x^{3/4}$.\\
	
	$3$. Let $E_{\alpha,\beta}$ be the elliptic curve defined by the equation $y^2=x^3+\alpha x+\beta$ with $\alpha$, $\beta\in \Z$. \Cref{langtrottersupersingularpairs} is known to hold on average for pairs of elliptic curves of the form $(E_{\alpha,\beta}, E_{\alpha',\beta'})$ due to work of Fouvry-Murty, see \cite{fouvrymurty}.\end{remarks}

\subsection{The height bound}\label{section:proofofhtbound}
In this subsection, we establish the height bounds announced in \Cref{heightboundintro} in the following, equivalent, form:
\begin{theorem}\label{mainhtbound}
	Let $f: X \rightarrow S$ be a $1$-parameter family of principally polarized abelian threefolds defined over a number field $K$ such that the image of the induced morphism $S \rightarrow A_{3}$ is a Hodge generic curve.

	Assume there exists $s_{0} \in S(K)$ for which $X_{s_{0}}$ is isogenous to $ E_{1} \times_K E_{2} \times_K E_{3}$ where $E_{1}$ is a CM elliptic curve and $E_{2}$, $E_{3}$ are non-CM, non-isogenous elliptic curves defined over $\mathbb{Q}$.
	
	If the Lang-Trotter conjecture holds for the pair $(E_{2},E_{3})$, then there exist effectively computable constants $c_{1}, c_{2}>0$ such that $h(s) \leq c_{1} \cdot[K(s): \mathbb{Q}]^{c_{2}}$ for all $s$ in the set
	$$
	\left\{s \in S(\bar{\Q}): X_{s} \text{ is isogenous to } E^{3}\text{ for some elliptic curve }E\right\} .
	$$
\end{theorem} 
\begin{proof}
	We have chosen to work in the simplified setting discussed in \Cref{section:backgroundgfuns}. This allows us to skip several technical difficulties and isolate the novel part of our argument. For more details we point the interested reader to the proof of Proposition $5.1$ in \cite{papaspadicpart1}.
	
	We thus consider a matrix of $G$-functions associated to the pair $\left(X \rightarrow S, s_{0}\right)$. We fix $s \in S(\bar{\Q})$ a point of the set in question and consider the set
	$$
	\Sigma(s):=\left\{v \in \Sigma_{K(s)}: s \text { is } v \text {-adically close to } s_{0}\right\} .
	$$
	
	By \Cref{lemmatypesofproximity}, $\Sigma(s)=\Sigma(s)_{\infty} \cup \Sigma(s)_{\ord} \cup \Sigma(s)_{\ssing}$, where $\Sigma(s)_{*}$ denotes the set of archimedean, respectively ordinary for $X_{0}$, respectively supersingular for $X_{0}$, places in $\Sigma(s)$.
	
	We define the polynomial
	$$
	R_{s}:=\left(\prod_{v \in \Sigma(s)_{\infty}} R_{s, v}\right) \cdot\left(\prod_{v \in \Sigma(s)_{\ssing}} R_{s, v}\right) \cdot R_{s, \ord},
	$$
	where $R_{s,\ord}$ is as in \Cref{propordinary} and $R_{s,v}$ is as in \Cref{proparchimedean} for $v \in \Sigma(s)_{\infty}$ and as in \Cref{propsupersing} for $v \in \Sigma(s)_{\ssing}$.
	
	By construction $R_s$ defines a ``global and non-trivial relation'' among the values of the family of G-functions $Y_{G}$ at $s$. By the André - Bombieri ``Hasse principle for values of $G$-functions'', see \cite[Ch. VII, \S5]{andre1989g}, there exist constants $c_{1}$, $c_{2}>0$ such that
	\begin{equation}\label{eq:htbound1}
			h(x(s)) \leqslant c_{1} \deg R_{s}^{c_{2}},
	\end{equation}
	for all such points.

	By construction, we have the trivial bound
	\begin{equation}\label{eq:degreehtboundproof1}
		\deg R_{s} \leq 2 \cdot[K(s): \mathbb{Q}]+2 \cdot\left|\Sigma(s)_{\ssing}\right|+3 .
	\end{equation}
	
	Recall that, in order to apply the results of Section \ref{section:relations}, we may have had to replace $K$ by a finite extension. In other words, our $E_2$ and $E_3$ above will be of the form $E_j=\CE_j\times_{\Q} K$, for elliptic curves $\CE_j$ defined over $\Q$. 
	
	While the $E_j$ will, by assumption, have everywhere semi-stable reduction, this is not necessarily true for the $\CE_j$. Therefore, let us set $N_j$ to be the conductor of $\CE_j$. Also, given $v\in\Sigma(s)_{\ssing}$, let us write $p(v):=\chara k(v)$, the characteristic of the residue field of $\CO_K$ at $v$. Then we can write $\Sigma(s)_{\ssing}$ as the disjoint union of the sets $\Sigma(s)_{\ssing,1}:=\{v\in\Sigma(s)_{\ssing}:p(v)|N_2 N_3\}$ and $\Sigma(s)_{\ssing,2}:=\{v\in\Sigma(s)_{\ssing}:p(v)\nmid N_2 N_3\}$. Using the trivial bound $|\Sigma(s)_{\ssing,1}|\leq [K(s):\Q](N_2N_3)$, we may rewrite \eqref{eq:degreehtboundproof1} as 
	\begin{equation}\label{eq:degreehtboundproof2}
		\deg R_{s} \leq (2+2N_2N_3) \cdot[K(s): \mathbb{Q}]+2 \cdot\left|\Sigma(s)_{\ssing,2}\right|+3.
	\end{equation}
	
	\begin{claim}There exist positive constants $\alpha_{0}$ and $c_3$, depending on $E_2$ and $E_3$, such that either 
	$\left|\Sigma(s)_{\ssing,2}\right| \leqslant c_{3}[K(s): \mathbb{Q}](\log [K(s): \mathbb{Q}]+\log h(x(s)))$
	or $h(x(s)) \leqslant \alpha_{0}$.
	\end{claim}
	
\begin{proof}[Proof of the claim]
		By the definition of ``v-adic proximity'' we have, for the subset $\Sigma(s)_f$ of $\Sigma(s)$ that consists of finite places,
		$$
		\Sigma(s)_f \subset \left\{v \in \Sigma_{K(s), f}:|x(s)|_{v}<1\right\}.
		$$
	
		Setting
			$$\Sigma(s)_{\mathbb{Q}}:=\left\{p \in \mathbb{Q}: p=p(v) \text{ for some }v \in \Sigma(s)_{f}\right\},$$
        by the definition of the Weil height, we have
		$$
	\sum_{p\in \Sigma(s)_{\Q}}^{}\log p\leq 	\sum_{p(v), |x(s)|_v<1 }^{}\log p(v) \leq [K(s):\Q] \cdot h(x(s)).
		$$
Thus, $p \leqslant e^{[K(s):\Q]\cdot h(x(s))}$ for all $p\in\Sigma(s)_{\mathbb{Q}}$. 
		
		Let $\alpha(s):= [K(s):\Q] \cdot h(x(s))$ and set 
		$$
		\Sigma(s)_{\ssing, \mathbb{Q}}:=\{p\in \Q: p=p(v) \text{ for some } v\in \Sigma(s)_{\ssing,2}\}.
		$$
		From the above we get 
		$$
		|\Sigma(s)_{\ssing,\Q}| \leq \pi_{\CE_2,\CE_3}(e^{\alpha(s)}),
		$$
		in the notation of \Cref{langtrottersupersingularpairs}.
		
		Since $\CE_{2}$ and $\CE_{3}$ are defined over $\Q$, \Cref{langtrottersupersingularpairs} implies that 
		$$
	\pi_{\CE_{2}, \CE_{3}}(x) \sim C_{\CE_{2}, \CE_{3}} \log \log x.
		$$
		In particular, for $x\geq x_0$, for some constant $x_0:=x_0(\CE_2,\CE_3)$, we have 
		$$
		\pi_{\CE_{2}, \CE_{3}}(x) \leqslant\left(C_{\CE_{2}, \CE_{3}}+1\right) \log \log x.
		$$
		If $\alpha(s)\leq x_0$, we deduce that $h(x(s))\leq\alpha_0$ for some $\alpha_0:=\alpha_0(\CE_2,\CE_3)$. Otherwise, 
		$$
		\left|\Sigma(s)_{\ssing,2}\right| \leq[K(s): \mathbb{Q}] \cdot\left|\Sigma(s)_{\ssing, \mathbb{Q}}\right| \leq\left(C_{\CE_{2}, \CE_{3}}+1\right)[K(s): \mathbb{Q}] \log \log e^{\alpha(s)}.
		$$
		The claim follows.
	\end{proof}

	If $h(x(s)) \leqslant \alpha_{0}$ then we are done. If not, we get from the claim a bound of the form
	$$
	\operatorname{deg} R_{s} \leqslant c_{4} \cdot[K(s): \mathbb{Q}]^{1+1 / 2} \cdot \log h(x(s)) .
	$$
	Combining this with \eqref{eq:htbound1} our result follows.\end{proof}
	
	\begin{remark}\label{remarkafterproof}
		We note here that the full strength of \Cref{langtrottersupersingularpairs} is slightly redundant in our case. Indeed, by the proof above a bound of the form $$\pi_{E_1,E_2}(x)\leq c^{\prime}_{E_1,E_2}( \log \log x)^{D_{E_1,E_2}}$$
		would be enough for our application. 
	\end{remark}
\subsection{Lang-Trotter for pairs over number fields}\label{section:langtrotternumberfield}
The condition that $E_2$ and $E_3$ are defined over $\Q$ in \Cref{mainthmzp} is a limitation that stems from our inability to find a reference for a version of \Cref{langtrottersupersingularpairs} for curves defined over arbitrary number fields in the literature. Such a conjecture over number fields seems natural to state based on heuristics coming from the Sato-Tate conjecture.

Let us fix a pair of non-CM, non-isogenous over $\bar{\Q}$ elliptic curves $E$ and $E'$ defined over a fixed number field $K$ with $d:=[K:\Q]$. Given a finite place $v$ of $K$, we write $\mathfrak{p}_v$ for the corresponding prime ideal of $\CO_K$, $p_v$ for the characteristic of the residue field $k_v:=\CO_K/\mathfrak{p}_v$, and $f_v$ for the degree of $\mathfrak{p}_v$, i.e. $p_v^{f_v}=|k_v|$.

Given a place $v$ as above over which both $E$ and $E'$ have good reduction it is classical, see for example \cite{silvermanell}, that for $X\in \{E,E'\}$ we have  \begin{equation}|X(k_v)|=p_v^{f_v}+1-a_X(v),\end{equation}where $a_X(v)$ is an integer satisfying Hasse's inequality $|a_X(v)|\leq 2 \sqrt{p_v^{f_v}}$. The Sato-Tate conjecture, for a single elliptic curve, predicts that the quantity $\frac{a_X(v)}{2 \sqrt{p_v^{f_v}}}$ equidistributes in the interval $[-1,1]$, see for example \cite{murtymurtysatotate}.

On the other hand, see for example the proof on page $150$ of \cite{silvermanell}, we know that $X\in \{E,E'\}$ will have supersingular reduction modulo $v$ if and only if $a_X(v)$ is divisible by the prime $p_v$. We are thus led to consider the condition that $\frac{a_X(v)}{2 \sqrt{p_v^{f_v}}}$ be in the set $\{0,n p_v  :1\leq |n| \leq 2 p_v^{\frac{f_v}{2}-1}\}$. Naively treating the Sato-Tate conjecture as a local density statement, we may think of the probability that $X\in \{E,E'\}$ is supersingular modulo $v$ as 
\begin{equation}\label{eq:satotatenaive1}
	\Prob(X \text{ is supersingular mod }v)\sim C(X)\frac{1}{p_v^{1/2}},
\end{equation}when $f_v=1$ and
\begin{equation}\label{eq:satotatenaive2}
	\Prob(X \text{ is supersingular mod }v)\sim C(X)\frac{p_v^{\frac{f_v}{2}-1}}{p_v^{f_v/2}}=C(X)\frac{1}{p_v},
\end{equation}when $f_v\geq 2$.

In the proof of \Cref{mainhtbound} the problematic set whose cardinality we want an upper bound for, is the set 
$$\Sigma_{\ssing}(E,E'):=\{ v\in\Sigma_{K,f}: \tilde{E}_v\text{, }\tilde{E}_v'\text{ are both supersingular}\}.$$
Following the proof of \Cref{mainhtbound} in the previous subsection, it is easy to see that it would suffice for our purposes to have a bound for the quantity
\begin{equation}\label{eq:langtrottergeneralizationnaive}
	\pi_{E,E',K}(x):=|\{p\leq x: \exists v\in \Sigma_{\ssing}(E,E')\text{ with }v|p \}|.
\end{equation}

The natural analogue of the Sato-Tate conjecture for the pair $(E,E')$ would then allow us to treat the probabilities that appear in \eqref{eq:satotatenaive1} and \eqref{eq:satotatenaive2} as independent. We note here that for $K=\Q$ this version of the Sato-Tate conjecture is a known theorem by work of M. Harris, see \cite{harrissatotatepairs}. In that regard, in the above naive sense, we may bound the probability that $p\leq x$ is in the set that appears in \eqref{eq:langtrottergeneralizationnaive} by \begin{equation}\label{eq:langtrotgenpair1}
	\Sum{v|p, f_v=1}{}C(E,E')\frac{1}{p}+\Sum{v|p, f_v\geq 2}{}C(E,E')\frac{1}{p^2}\leq d C(E,E')\frac{1}{p}.
	\end{equation}
Summing this over all $p\leq x$ would thus lead us to 
\begin{equation}
		\pi_{E,E',K}(x)\leq d C(E,E')\Sum{p\leq x}{}\frac{1}{p}.
\end{equation}The latter sum is known to be of the order $\log\log(x)$ thanks to Merten's second theorem.

All in all, it seems natural based on the above heuristic to state the following:
\begin{conj}\label{conjonsupersingred}[Conjecture on simultaneous supersingular reduction]
	Let $E$ and $E'$ be two elliptic curves defined over a number field $K$. Assume that $E$ and $E'$ are non-CM and non-isogenous over $\bar{\Q}$. Then there exists a constant $C(E,E',K)\geq 0$ for which
	\begin{equation}\label{eq:langtrotconjgenforpairs}
		\pi_{E,E',K}(x)\leq C(E,E',K) \log\log(x),
	\end{equation}as $x\rightarrow\infty$.
\end{conj}
\begin{remarks}$1$. While our treatment of Sato-Tate is heuristic, \Cref{conjonsupersingred} appears to be consistent with Lang-Trotter statements over number fields that appear in the literature for a single elliptic curve. See for example the survey \cite{jamessurvey}.\\
	
	$2$. Replacing the use of \Cref{langtrottersupersingularpairs} in the proof of \Cref{mainhtbound} by \Cref{conjonsupersingred} would allow us to remove the condition in \Cref{mainhtbound}, and hence in \Cref{mainthmzp}, that $E_2$ and $E_3$ be defined over $\Q$.\\
	
	$3$. As noted in \Cref{remarkafterproof}, the bound in \eqref{eq:langtrotconjgenforpairs} is slightly stronger than is required for our purposes. In fact, an upper bound in \eqref{eq:langtrotconjgenforpairs} of the form $ \poly_{E,E',K}(\log \log(x))$ would suffice, where by $\poly_{E,E',K}$ we mean a polynomial whose coefficients and degree depend on $E$, $E'$, and $K$. This version of the argument is carried out explicitly in the proof of the height bounds in \cite{papaszpy1note}.
	\end{remarks}

\subsection{Beyond \Cref{mainthmzp}}\label{secion:speculation}

\Cref{mainhtbound}, and by extension \Cref{mainthmzp}, should be viewed as a special case of a more general phenomenon.

Consider a smooth Hodge generic and irreducible curve $S\subset \mathcal{A}_g$ which is defined over $\bar{\Q}$, $s_0\in S(\bar{\Q})$, and write $A_0$ for the $g$-dimensional abelian variety corresponding to $s_0$. We assume that $s_0$ lies on a special subvariety $Z_0$ of sufficiently large codimension $r_0$, which we assume to be minimal for $A_0$. Let us also fix some subset $\Sha(S)$ of $S(\bar{\Q})$ that is contained in the locus of intersections of $S$ with special subvarieties of codimension at least $2$, i.e. some subset of the points pertinent to the Zilber-Pink conjecture.

As in \Cref{section:backgroundgfuns} we associate to $(S,s_0)$ a family of G-functions. Fix from now on a point $s\in \Sha(S)$. As in the proof of \Cref{mainhtbound} our goal now is to construct a global and non-trivial relation $R_s$ among the values $Y_G(s)$ and establish bounds of the form $\deg(R_s)\ll [\Q(s):\Q]^c$ for some constant $c$ that does not depend on $s$.

In short, there are two crucial phenomena that we expect to hold. The first is that upon assuming that $r_0$ is large enough, the number of possible ``local factors'' $R_{s,v}$ that make up the polynomial $R_s$ is bounded by a constant independent of $s$ as long as $v$ is not a place of supersingular reduction of the ``central fiber'' $A_0$. In our particular case, this is largely the reason for our choice of $s_0$, reflected in \Cref{propordinary}.

The second phenomenon that we expect here is that, upon assuming that $r_0$ is not ``too large'', the number of places of supersingular reduction of $A_0$ should be rather sparse. In our case, this sparsity is controlled by \Cref{langtrottersupersingularpairs}.

\appendix	
\section{Appendix}\label{section:appendixmathematica}
In \Cref{proparchimedean}, \Cref{propsupersing}, and \Cref{propordinary} we had to use a membership test for the polynomials we produced in the ideal $I(\SP_6)$. To do this we used the code in Wolfram Mathematica that we have included here.

\subsection{The base code}\label{section:basecode}
This is the main computational part of the code. The first code below computes the polynomials in question and saves each to a separate file to be used in the following steps.

{\small\begin{doublespace}
	\noindent\({\text{ClearAll}[\text{{``}Global$\grave{ }$*{''}}]}\\
	{\text{(*}\text{Define} \text{ submatrices} \text{ for } \text{{``}de Rham isogenies{''}}.\text{*)}}\\
	{\text{As}=\{\{\text{d11},\text{d12},\text{d13}\},\{\text{d21},\text{d22},\text{d23}\},\{\text{d31},\text{d32},\text{d33}\}\};}\\
	{\text{Bs}=\{\{\text{f11},\text{f12},\text{f13}\},\{\text{f21},\text{f22},\text{f23}\},\{\text{f31},\text{f32},\text{f33}\}\};}\\
	{\text{Cs}=\{\{\text{e11},\text{e12},\text{e13}\},\{\text{e21},\text{e22},\text{e23}\},\{\text{e31},\text{e32},\text{e33}\}\};}\\
	{\text{A0}=\{\{1,0,0\},\{0,1,0\},\{0,0,1\}\};}\\
	{\text{B0}=\{\{\text{b11},\text{b12},\text{b13}\},\{\text{b21},\text{b22},\text{b23}\},\{\text{b31},\text{b32},\text{b33}\}\};}\\
	{\text{C0}=\{\{\text{c11},\text{c12},\text{c13}\},\{\text{c21},\text{c22},\text{c23}\},\{\text{c31},\text{c32},\text{c33}\}\};}\\
	{\text{(*Define symbolic constants*)}}\\
	{\text{d1}=\text{d1};\text{d2}=\text{d2};}\\
	{\text{(*Define the symbolic matrix Y*)}}\\
	{Y=\{\{\text{X11},\text{X12},\text{X13},\text{X14},\text{X15},\text{X16}\},\{\text{X21},\text{X22},\text{X23},\text{X24},\text{X25},\text{X26}\},}\\
	{\{\text{X31},\text{X32},\text{X33},\text{X34},\text{X35},\text{X36}\},\{\text{X41},\text{X42},\text{X43},\text{X44},\text{X45},\text{X46}\},}\\
	{\{\text{X51},\text{X52},\text{X53},\text{X54},\text{X55},\text{X56}\},\{\text{X61},\text{X62},\text{X63},\text{X64},\text{X65},\text{X66}\}\};}\\
	{\text{(*Define the block matrices of the isogenies and the constant matrix J*)}}\\
	{M=\text{ArrayFlatten}[\{\{\text{As},\text{ConstantArray}[0,\{3,3\}]\},\{\text{Bs},\text{Cs}\}\}];}\\
	{\text{NMatrix}=\text{ArrayFlatten}[\{\{\text{A0},\text{ConstantArray}[0,\{3,3\}]\},\{\text{B0},\text{C0}\}\}];}\\
	{J=\{\{1,0,0,0,0,0\},\{0,0,0,1,0,0\},\{0,1,0,0,0,0\},}\\
	{\{0,0,0,0,1,0\},\{0,0,1,0,0,0\},\{0,0,0,0,0,1\}\};}\\
	{\text{J1}=\{\{1,0,0,0,0,0\},\{0,0,1,0,0,0\},\{0,0,0,0,1,0\},}\\
	{\{0,1,0,0,0,0\},\{0,0,0,1,0,0\},\{0,0,0,0,0,1\}\};}\\
	{\text{(*Compute the matrix in steps*)}}\\
	{\text{Finter0}=\text{Simplify}[M.Y];\text{Finter1}=\text{Simplify}[\text{Finter0}.\text{NMatrix}];}\\
	{\text{Finter2}=\text{Simplify}[\text{Finter1}.J];\text{Pmatrix}=\text{Simplify}[\text{J1}.\text{Finter2}];}\\
	{\text{(*Extract the elements of Pmatrix*)}}\\
	{\text{P11}=\text{Pmatrix}[[1,1]];\text{P12}=\text{Pmatrix}[[1,2]];\text{P13}=\text{Pmatrix}[[1,3]];\text{P14}=\text{Pmatrix}[[1,4]];}\\
	{\text{P15}=\text{Pmatrix}[[1,5]];\text{P16}=\text{Pmatrix}[[1,6]];}\\
	{\text{P21}=\text{Pmatrix}[[2,1]];\text{P22}=\text{Pmatrix}[[2,2]];\text{P23}=\text{Pmatrix}[[2,3]];\text{P24}=\text{Pmatrix}[[2,4]];}\\
	{\text{P25}=\text{Pmatrix}[[2,5]];\text{P26}=\text{Pmatrix}[[2,6]];}\\
	{\text{P31}=\text{Pmatrix}[[3,1]];\text{P32}=\text{Pmatrix}[[3,2]];\text{P33}=\text{Pmatrix}[[3,3]];\text{P34}=\text{Pmatrix}[[3,4]];}\\
	{\text{P35}=\text{Pmatrix}[[3,5]];\text{P36}=\text{Pmatrix}[[3,6]];}\\
	{\text{P41}=\text{Pmatrix}[[4,1]];\text{P42}=\text{Pmatrix}[[4,2]];\text{P43}=\text{Pmatrix}[[4,3]];\text{P44}=\text{Pmatrix}[[4,4]];}\\
	{\text{P45}=\text{Pmatrix}[[4,5]];\text{P46}=\text{Pmatrix}[[4,6]];}\\
	{\text{P51}=\text{Pmatrix}[[5,1]];\text{P52}=\text{Pmatrix}[[5,2]];\text{P53}=\text{Pmatrix}[[5,3]];\text{P54}=\text{Pmatrix}[[5,4]];}\\
	{\text{P55}=\text{Pmatrix}[[5,5]];\text{P56}=\text{Pmatrix}[[5,6]];}\\
	{\text{P61}=\text{Pmatrix}[[6,1]];\text{P62}=\text{Pmatrix}[[6,2]];\text{P63}=\text{Pmatrix}[[6,3]];\text{P64}=\text{Pmatrix}[[6,4]];}\\
	{\text{P65}=\text{Pmatrix}[[6,5]];\text{P66}=\text{Pmatrix}[[6,6]];}\\
	{\text{(*The relations we need to check*)}}\\
	{\text{Rarch}=\text{P11}*\text{P22}-\text{P12}*\text{P21}-\text{d1};}\\
	{\text{Rssing}=\text{P31}*\text{P42}-\text{P32}*\text{P41}-\text{d2}*(\text{P11}*\text{P22}-\text{P12}*\text{P21});}\\
	{\text{Rord1}=\text{P22}*\text{P32}-\text{P12}*\text{P42};\text{(*This is A12*)}}\\
	{\text{(*}\text{Expand} \text{ the } \text{polynomials}.\text{*)}}\\
	{\text{expRarch}=\text{Expand}[\text{Rarch}];\text{expRssing}=\text{Expand}[\text{Rssing}];\text{expRord1}=\text{Expand}[\text{Rord1}];}\\
	{\text{(*Save each to a separate file*)}}\\
	{\text{DumpSave}[\text{{``}archimedeanrelA3.mx{''}},\text{expRarch}];\text{DumpSave}[\text{{``}ssingrelA3.mx{''}},\text{expRssing}];}\\
	{\text{DumpSave}[\text{{``}ordinaryA31.mx{''}},\text{expRord1}];}\\
	{\text{(*Output a confirmation message*)}}\\
	{\text{Print}[\text{{``}Quantities saved to specified files.{''}}];}\)
\end{doublespace}}
\subsubsection{Computing a Gr\"obner basis}\label{section:grobnerbasis}
The second reusable code outputs a Gr\"obner basis of the ideal $I(\SP_6)$ and saves it in a separate file to be used in the membership tests.

{\small\begin{doublespace}
	\noindent\({\text{ClearAll}[\text{{``}Global$\grave{ }$*{''}}];}\\
	{\text{(*Define variables*)}}\\
	{\text{vars}=\{\text{X11},\text{X12},\text{X13},\text{X14},\text{X15},\text{X16},\text{X21},\text{X22},\text{X23},\text{X24},\text{X25},\text{X26},}\\
	{\text{X31},\text{X32},\text{X33},\text{X34},\text{X35},\text{X36},\text{X41},\text{X42},\text{X43},\text{X44},\text{X45},\text{X46},\text{X51},\text{X52},\text{X53},\text{X54},\text{X55},}\\
	{\text{X56},\text{X61},\text{X62},\text{X63},\text{X64},\text{X65},\text{X66}\};}\\
	{\text{(*Build the ``generic'' 6x6 matrix*)}}\\
	{X=\text{Partition}[\text{vars},6];}\\
	{\text{(*Standard symplectic form*)}}\\
	{J=\text{ArrayFlatten}[\{\{\text{ConstantArray}[0,\{3,3\}],\text{IdentityMatrix}[3]\},}\\
	{\{-\text{IdentityMatrix}[3],\text{ConstantArray}[0,\{3,3\}]\}\}];}\\
	{\text{(*The 15 independent quadratic generators*)}}\\
	{\text{gens}=\text{Flatten}@\text{Table}[\text{Expand}[(\text{Transpose}[X].J.X-J)[[i,j]]],\{i,1,6\},\{j,i+1,6\}];}\\
	{\text{(*Compute the Gr{\" o}bner basis*)}}\\
	{\text{groebnerBasisA3}=\text{GroebnerBasis}[\text{gens},\text{vars},\text{MonomialOrder}\text{-$>$}\text{DegreeReverseLexicographic}];}\\
	{\text{(*Save generators and basis*)}}\\
	{\text{DumpSave}[\text{{``}groebnerbasis$\_$sp6.mx{''}},\{\text{gens},\text{groebnerBasisA3}\}];}\\
	{\text{Print}[\text{{``}The generators and Gr{\" o}bner basis have been saved to groebnerbasis$\_$sp6.mx{''}}];}\)
\end{doublespace}}
\subsection{Membership tests}\label{section:appmembership}
The codes used for the membership tests of the three polynomials are listed below.

\subsubsection{Archimedean places}\label{section:apparchimedean}
{\small\begin{doublespace}
	\noindent\({\text{ClearAll}[\text{{``}Global$\grave{ }$*{''}}];}\\
	{\text{(*Load the output from the previous codes*)}}\\
	{\text{Get}[\text{{``}archimedeanrelA3.mx{''}}] ;\text{Get}[\text{{``}groebnerbasis$\_$sp6.mx{''}}];}\\
	{\text{(*Define vars to include only polynomial variables*)}}\\
	{\text{vars}=\{\text{X11},\text{X12},\text{X13},\text{X14},\text{X15},\text{X16},\text{X21},\text{X22},\text{X23},\text{X24},\text{X25},\text{X26},}\\
	{\text{X31},\text{X32},\text{X33},\text{X34},\text{X35},\text{X36},\text{X41},\text{X42},\text{X43},\text{X44},}\\
	{\text{X45},\text{X46},\text{X51},\text{X52},\text{X53},\text{X54},\text{X55},\text{X56},\text{X61},\text{X62},\text{X63},\text{X64},\text{X65},\text{X66}\};}\\
	{\text{(*Ensure constants are treated as symbolic coefficients*)}}\\
	{\text{SetAttributes}[\{\text{d11},\text{d12},\text{d13},\text{d21},\text{d22},\text{d23},\text{d31},\text{d32},\text{d33},\text{f11},\text{f12},\text{f13},\text{f21},}\\
	{\text{f22},\text{f23},\text{f31},\text{f32},\text{f33},\text{b11},\text{b12},\text{b13},\text{b21},\text{b22},\text{b23},\text{b31},\text{b32},\text{b33},}\\
	{\text{c11},\text{c12},\text{c13},\text{c21},\text{c22},\text{c23},\text{c31},\text{c32},\text{c33},\text{e11},\text{e12},\text{e13},\text{e21},\text{e22},}\\
	{\text{e23},\text{e31},\text{e32},\text{e33},\text{d1},\text{d2}\},\text{Constant}];}\\
	{\text{(*Compute the polynomial reduction with respect to the Gr{\" o}bner basis*)}}\\
	{\text{redRarch}=\text{PolynomialReduce}[\text{expRarch},\text{groebnerBasisA3},\text{vars}];}\\
	{\text{remRarch}=\text{Last}[\text{redRarch}];}\\
	{\text{(*Extract coefficients and monomials of Rarch*)}}\\
	{\text{moncoeffRarch}=\text{CoefficientRules}[\text{remRarch},\text{vars}];}\\
	{\text{(*}\text{Format the result as a list} \text{ with two columns}:\text{monomials} \text{ and } \text{coefficients}\text{*)}}\\
	{\text{ListRarch}=\text{Table}[\{\text{Times}\text{@@}(\text{vars}{}^{\wedge}\text{rule}[[1]]),\text{rule}[[2]]\},\{\text{rule},\text{moncoeffRarch}\}];}\\
	{\text{(*Define a function to process one chunk*)}}\\
	{\text{processChunk}[\text{chunk$\_$}]\text{:=}\text{Table}[\{\text{entry}[[1]],\text{Factor}[\text{entry}[[2]]]\},\{\text{entry},\text{chunk}\}];}\\
	{\text{(*Set chunk size*)}}\\
	{\text{chunkSize}=100; \text{(*}\text{Adjust based on system's capability}\text{*)}}\\
	{\text{(*Break the list into chunks*)}}\\
	{\text{chunksRarch}=\text{Partition}[\text{ListRarch},\text{chunkSize},\text{chunkSize},1,\{\}];}\\
	{\text{(*Process each chunk and collect results*)}}\\
	{\text{finalListRarch}=\text{Flatten}[\text{processChunk}[\#]\&\text{/@}\text{chunksRarch},1];}\\
	{\text{(*Save the factored list*)}}\\
	{\text{DumpSave}[\text{{``}finalListArchimedeanA3.mx{''}},\{\text{finalListRarch}\}];}\\
	{\text{(*Output a confirmation message*)}}\\
	{\text{Print}[\text{{``}Factored list saved to specified mx file.{''}}];}\)
\end{doublespace}}

\subsubsection{Supersingular places}\label{section:appsupersing}

{\small\begin{doublespace}
	\noindent\({\text{ClearAll}[\text{{``}Global$\grave{ }$*{''}}];}\\
	{\text{(*Load the output from the previous codes*)}}\\
	{\text{Get}[\text{{``}ssingrelA3.mx{''}}] ;\text{Get}[\text{{``}groebnerbasis$\_$sp6.mx{''}}];}\\
	{\text{(*The rest is done as in the previous code*)}}\\
	{\text{vars}=\{\text{X11},\text{X12},\text{X13},\text{X14},\text{X15},\text{X16},\text{X21},\text{X22},\text{X23},\text{X24},\text{X25},\text{X26},}\\
{\text{X31},\text{X32},\text{X33},\text{X34},\text{X35},\text{X36},\text{X41},\text{X42},\text{X43},\text{X44},}\\
{\text{X45},\text{X46},\text{X51},\text{X52},\text{X53},\text{X54},\text{X55},\text{X56},\text{X61},\text{X62},\text{X63},\text{X64},\text{X65},\text{X66}\};}\\
	{\text{SetAttributes}[\{\text{d11},\text{d12},\text{d13},\text{d21},\text{d22},\text{d23},\text{d31},\text{d32},\text{d33},\text{f11},\text{f12},\text{f13},\text{f21},}\\
{\text{f22},\text{f23},\text{f31},\text{f32},\text{f33},\text{b11},\text{b12},\text{b13},\text{b21},\text{b22},\text{b23},\text{b31},\text{b32},\text{b33},}\\
{\text{c11},\text{c12},\text{c13},\text{c21},\text{c22},\text{c23},\text{c31},\text{c32},\text{c33},\text{e11},\text{e12},\text{e13},\text{e21},\text{e22},}\\
{\text{e23},\text{e31},\text{e32},\text{e33},\text{d1},\text{d2}\},\text{Constant}];}\\
	{\text{redRssing}=\text{PolynomialReduce}[\text{expRssing},\text{groebnerBasisA3},\text{vars}];}\\
	{\text{remRssing}=\text{Last}[\text{redRssing}];}\\
	{\text{moncoeffRssing}=\text{CoefficientRules}[\text{remRssing},\text{vars}];}\\
	{\text{ListRssing}=\text{Table}[\{\text{Times}\text{@@}(\text{vars}{}^{\wedge}\text{rule}[[1]]),\text{rule}[[2]]\},\{\text{rule},\text{moncoeffRssing}\}];}\\
	{\text{processChunk}[\text{chunk$\_$}]\text{:=}\text{Table}[\{\text{entry}[[1]],\text{Factor}[\text{entry}[[2]]]\},\{\text{entry},\text{chunk}\}];}\\
	{\text{chunkSize}=100; \text{(*}\text{Adjust based on system's capability}\text{*)}}\\
	{\text{chunksRssing}=\text{Partition}[\text{ListRssing},\text{chunkSize},\text{chunkSize},1,\{\}];}\\
	{\text{finalListRssing}=\text{Flatten}[\text{processChunk}[\#]\&\text{/@}\text{chunksRssing},1];}\\
	{\text{DumpSave}[\text{{``}finalListsupersingularA3.mx{''}},\{\text{finalListRssing}\}];}\\
	{\text{Print}[\text{{``}Factored list saved to specified mx file.{''}}];}\)
\end{doublespace}}

\subsubsection{Ordinary places}\label{section:appordinary}
{\small\begin{doublespace}
	\noindent\({\text{ClearAll}[\text{{``}Global$\grave{ }$*{''}}];}\\
	{\text{(*Load the output from the previous codes*)}}\\
	{\text{Get}[\text{{``}ordinaryA31.mx{''}}] ;\text{Get}[\text{{``}groebnerbasis$\_$sp6.mx{''}}];}\\
{\text{(*The rest is done as in the previous codes*)}}\\
{\text{vars}=\{\text{X11},\text{X12},\text{X13},\text{X14},\text{X15},\text{X16},\text{X21},\text{X22},\text{X23},\text{X24},\text{X25},\text{X26},}\\
{\text{X31},\text{X32},\text{X33},\text{X34},\text{X35},\text{X36},\text{X41},\text{X42},\text{X43},\text{X44},}\\
{\text{X45},\text{X46},\text{X51},\text{X52},\text{X53},\text{X54},\text{X55},\text{X56},\text{X61},\text{X62},\text{X63},\text{X64},\text{X65},\text{X66}\};}\\
{\text{SetAttributes}[\{\text{d11},\text{d12},\text{d13},\text{d21},\text{d22},\text{d23},\text{d31},\text{d32},\text{d33},\text{f11},\text{f12},\text{f13},\text{f21},}\\
{\text{f22},\text{f23},\text{f31},\text{f32},\text{f33},\text{b11},\text{b12},\text{b13},\text{b21},\text{b22},\text{b23},\text{b31},\text{b32},\text{b33},}\\
{\text{c11},\text{c12},\text{c13},\text{c21},\text{c22},\text{c23},\text{c31},\text{c32},\text{c33},\text{e11},\text{e12},\text{e13},\text{e21},\text{e22},}\\
{\text{e23},\text{e31},\text{e32},\text{e33},\text{d1},\text{d2}\},\text{Constant}];}\\
	{\text{redRord1}=\text{PolynomialReduce}[\text{expRord1},\text{groebnerBasisA3},\text{vars}];}\\
	{\text{remRord1}=\text{Last}[\text{redRord1}];}\\
	{\text{moncoeffRord1}=\text{CoefficientRules}[\text{remRord1},\text{vars}];}\\
	{\text{ListRord1}=\text{Table}[\{\text{Times}\text{@@}(\text{vars}{}^{\wedge}\text{rule}[[1]]),\text{rule}[[2]]\},\{\text{rule},\text{moncoeffRord1}\}];}\\
	{\text{processChunk}[\text{chunk$\_$}]\text{:=}\text{Table}[\{\text{entry}[[1]],\text{Factor}[\text{entry}[[2]]]\},\{\text{entry},\text{chunk}\}];}\\
	{\text{chunkSize}=100; \text{(*}\text{Adjust based on system's capability}\text{*)}}\\
	{\text{chunksRord1}=\text{Partition}[\text{ListRord1},\text{chunkSize},\text{chunkSize},1,\{\}];}\\
	{\text{finalListRord1}=\text{Flatten}[\text{processChunk}[\#]\&\text{/@}\text{chunksRord1},1];}\\
	{\text{DumpSave}[\text{{``}finalListordA31.mx{''}},\{\text{finalListRord1}\}];}\\
	{\text{Print}[\text{{``}Factored list saved to specified mx file.{''}}];}\)
\end{doublespace}}
\bibliographystyle{alpha}
\bibliography{biblio}
\end{document}